\documentclass[12pt,twoside]{amsart}

\usepackage{geometry}
\usepackage{amssymb, amsmath, amsthm, amscd}
\usepackage{enumerate}
\usepackage{mathrsfs}
\usepackage{bm}

\usepackage{graphicx}
\usepackage{hhline}
\usepackage[all]{xy}

\usepackage{xcolor}

\usepackage{url}

\usepackage{iftex}

\ifPDFTeX

  \usepackage[pdfencoding=auto, unicode]{hyperref}
\else
  \usepackage[dvipdfmx, unicode]{hyperref}
\fi

\hypersetup{
    colorlinks=false,
    linkcolor=blue,
    citecolor=blue,
    urlcolor=blue,
    bookmarksnumbered=true,
    pdfstartview=FitH
}

\title[A characterization of projective space]
{A characterization of projective space via lengths of extremal rays} 

\subjclass[2020]{Primary 14E30; Secondary 14J17}
\keywords{lengths of extremal rays, cone theorem, Mori theory, 
projective space}

\author{Osamu Fujino}
\address{Department of 
Mathematics, Graduate School of Science, 
Kyoto University, Kyoto 606-8502, Japan}
\email{fujino@math.kyoto-u.ac.jp}

\author{Eric Jovinelly}
\address{Department of Mathematics \\
Brown University \\
Box 1917 \\
151 Thayer Street \\
Providence, RI, 02912}
\email{eric\_jovinelly@brown.edu}

\author{Brian Lehmann}
\address{Department of Mathematics \\
Boston College  \\
Chestnut Hill, MA \, \, 02467}
\email{lehmannb@bc.edu}

\author{Eric Riedl}
\address{Department of Mathematics \\
University of Notre Dame  \\
255 Hurley Hall \\
Notre Dame, IN 46556}
\email{eriedl@nd.edu}

\DeclareMathOperator{\NE}{\overline{NE}}
\DeclareMathOperator{\Nlc}{Nlc}
\DeclareMathOperator{\Nklt}{Nklt}

\newtheorem{thm}{Theorem}[section]
\newtheorem{lem}[thm]{Lemma}
\newtheorem{cor}[thm]{Corollary}

\theoremstyle{definition}
\newtheorem{defn}[thm]{Definition}
\newtheorem{rem}[thm]{Remark}

\newtheorem*{ack}{Acknowledgments}

\makeatletter

\@addtoreset{equation}{section}
\makeatother

\begin{document}

\begin{abstract}
We prove a new characterization of complex projective space using lengths of extremal rays.
\end{abstract}

\maketitle 


\section{Introduction}\label{a-sec1}

We work over the complex 
number field $\mathbb{C}$.  In this short note, our goal is to give a new characterization of projective space which relies on the anticanonical degrees of curves.  The first statements in this direction are due to the influential paper \cite{Mori79}.  \cite{CMSB} and \cite{Kebekus02} extended Mori's argument to show that $\mathbb{P}^{n}$ is the only smooth projective variety of dimension $n$ such that every curve has anticanonical degree $\geq (n+1)$.  \cite{CT07} further improved this characterization of $\mathbb{P}^{n}$ by allowing isolated LCIQ singularities and proving that one only needs to find a single extremal ray of length $\geq n+1$. \cite{Chen17} addressed varieties with quotient singularities.

\medskip

Our main result extends previous work on characterizations of $\mathbb{P}^{n}$ in several directions: it covers all MMP-type singularities, weakens the assumption on the anticanonical degree to the lowest possible value, and only requires knowledge of a single ray of the Mori cone.  Thus it also contributes to the study of varieties with ``long'' extremal rays, which has been carried out in the smooth setting by \cite{AO02, HN13, DH17}. 

\begin{thm}\label{thm:maintheorem}
Let $X$ be an $n$-dimensional projective variety and let 
$\Delta$ be an effective $\mathbb{R}$-divisor on $X$ such that 
$K_X+\Delta$ is $\mathbb{R}$-Cartier. 
Assume that there exists a $(K_X+\Delta)$-negative extremal ray
\[
R \subset \NE(X)
\]
which is rational, relatively ample at infinity, and satisfies
\[
l(R):=\inf\bigl\{-(K_X+\Delta)\cdot C \, | \, C \textrm{ is rational and }[C] \in R\bigr\} > n.
\]
Then $X\simeq \mathbb P^n$, $(X,\Delta)$ is terminal, and $\Delta$ is a divisor of degree less than $1$.
\end{thm}

This result was known previously for toric pairs $(X, \Delta)$ by \cite{fujino1} (and generalized to toric foliations in
\cite{FS24} and \cite{FS25}).  The consequences of Theorem \ref{thm:maintheorem} for log canonical pairs are particularly striking.  In this setting, the quantity $l(R)$ of Theorem \ref{thm:maintheorem} is known as the length of the extremal ray $R$ (see Remark \ref{b-rem3.3}).

\begin{cor}\label{cor:maincor1}
Let $(X, \Delta)$ be an $n$-dimensional 
projective log canonical pair such that $\NE(X)$ admits a $(K_{X}+\Delta)$-negative extremal ray of length $>n$.  Then $X\simeq \mathbb P^n$, $(X,\Delta)$ is terminal, and $\Delta$ is a divisor of degree less than $1$.
\end{cor}

\begin{cor}\label{coro:maincor2}
Let $(X, \Delta)$ be an $n$-dimensional 
projective log canonical pair such that $K_{X} + \Delta$ is not nef.  Suppose that
\[
-(K_X+\Delta)\cdot C> n
\] 
for every rational curve $C$ on $X$. Then $X\simeq \mathbb P^n$, $(X,\Delta)$ is terminal, and $\Delta$ is a divisor of degree less than $1$.
\end{cor}

There are similar statements for semi-log canonical pairs; for example, Corollary \ref{cor-slc2} establishes a semi-log canonical version of Corollary \ref{cor:maincor1}.

\begin{rem}
    A celebrated result of Kobayashi and Ochiai characterizes the smooth Fano varieties of large index (see e.g.~\cite{KO82, Ionescu86, Fujita87} and \cite{AD14,FM21} for generalizations to worse singularities).  One might hope for an analogous characterization of varieties with ``long'' extremal rays.  In the smooth setting this is accomplished in \cite{DH17} and in the toric setting by \cite{Fujino06}; it is natural to ask whether this classification can be extended to the singular setting.
\end{rem}

Our proof of Theorem \ref{thm:maintheorem} relies on several ingredients.  First, we need a classification result from \cite{CMSB}.  Second, we use the recently improved bounds in Bend-and-Break established by \cite{JLR25a} and the existence of very free curves on mildly singular Fano varieties proved by \cite{JLR25b}.  Third, we require the theory of quasi-log structures originally introduced by Ambro in \cite{Ambro03}; see \cite[Chapter~6]{fujino3} and \cite{fujino5} for more background.  More precisely, we will need the theory of extremal rays for quasi-log schemes as developed by \cite{fujino6}.

In Section \ref{a-sec2} we present several results on the cone of curves for quasi-log schemes obtained by combining the results of \cite{fujino6} and \cite{JLR25a}.  Throughout we adopt the notation and results established in 
\cite{fujino2}, \cite{fujino3}, \cite{fujino5}, and \cite{fujino6}.  In Section \ref{b-sec3} we prove Theorem \ref{thm:maintheorem} and related results.

\medskip

\begin{ack}
Osamu Fujino was partially supported by JSPS KAKENHI Grant Numbers
JP21H04994 and JP23K20787.  Eric Jovinelly was supported by an NSF postdoctoral research fellowship, DMS-2303335. Brian Lehmann was supported by Simons Foundation grant Award Number 851129.  Eric Riedl was supported by NSF CAREER grant DMS-1945944 and Simons Foundation grants 00011850 and 00013673.

The authors thank Makoto Enokizono and Kento Fujita for useful discussions.
\end{ack}

\section{Extremal rays for quasi-log schemes}\label{a-sec2}

Classical studies on the lengths of extremal rays,  
for example \cite{Kawamata91}, rely on the Bend-and-Break theorem of \cite{miyaoka-mori}.  The bounds in these results can be improved by replacing \cite[Theorem 5]{miyaoka-mori} with \cite[Theorem 1.1]{JLR25a}.  This is equally true for the statements on extremal rays in the quasi-log setting presented below.  We emphasize that such generalizations are not merely of technical interest but are quite useful, e.g.~in proving Theorem \ref{thm:maintheorem} and Corollary \ref{cor:maincor1}.

\begin{thm}[{\cite[Theorem 1.12]{fujino6}}]\label{a-thm1.2}
Let $[X,\omega]$ be a quasi-log scheme and 
let $\varphi\colon X\to W$ be a projective 
morphism of schemes such that $-\omega$ is $\varphi$-ample. 
Let $P$ be an arbitrary closed point of $W$. 
Let $E$ be any positive-dimensional 
irreducible component of $\varphi^{-1}(P)$ 
such that $E\not\subset X_{-\infty}$. 
Then $E$ is covered by {\em{(}}possibly singular{\em{)}} rational curves 
$\ell$ with
\[0<-\omega\cdot \ell \leq \dim E+1. 
\]
In particular, $E$ is uniruled.
\end{thm}

\begin{proof}
In the final step of the proof of \cite[Theorem 1.12]{fujino6} (see 
\cite[p.~676]{fujino6} for details), we use 
\[
\begin{split}
0 < -\nu^*\omega \cdot \Gamma 
&\leq 2 \dim \overline{E} \cdot 
\frac{-\nu^*\omega \cdot C}{-K_{\overline{E}} \cdot C} \\
&\leq 2 \dim \overline{E},
\end{split}
\]
which essentially goes back to \cite{miyaoka-mori}.  
By applying \cite[Theorem 1.1]{JLR25b} instead, we can replace  
$2 \dim \overline{E}$ with $\dim \overline{E} + 1$.  
This yields the desired inequality in Theorem~\ref{a-thm1.2}. 
\end{proof}

\begin{thm}[{Lengths of extremal rational curves, 
\cite[Theorem 1.13]{fujino6}}]\label{a-thm1.3}
Let \([X,\omega]\) be a quasi-log scheme, and let 
\(\pi \colon X \to S\) be a projective morphism of schemes.
Suppose that 
\(R \subset \NE(X/S)\) 
is an \(\omega\)-negative extremal ray which is rational 
and relatively ample at infinity. 
Let $\varphi_R\colon X\to W$ be the contraction morphism 
over $S$ associated to $R$. We put
\[d = \min_E\dim E, 
\]
where $E$ runs over positive-dimensional 
irreducible components of $\varphi^{-1}_R(P)$ 
for all $P\in W$. Then $R$ is 
spanned by a {\em{(}}possibly singular{\em{)}} rational curve $\ell$ with
\[
0 < -\omega \cdot \ell \leq d+1.
\]
\end{thm}

\begin{proof}
The proof of \cite[Theorem~1.13]{fujino6} applies verbatim.  
The only change is to replace \cite[Theorem~1.12]{fujino6}  
with Theorem~\ref{a-thm1.2} in the proof of \cite[Theorem~1.13]{fujino6} (see 
\cite[p.~676]{fujino6} for details). 
\end{proof}

\begin{cor}[{\cite[Corollary 12.3]{fujino6}}]\label{a-cor1.4} 
Let $X$ be a normal variety and let $\Delta$ be an 
effective $\mathbb R$-divisor on $X$ such that $K_X+\Delta$ is $\mathbb R$-Cartier. 
Let $\pi\colon X\to S$ be a projective morphism of 
schemes. Let $R$ be a $(K_X+\Delta)$-negative extremal ray of 
$\NE(X/S)$ that is 
rational and relatively ample at infinity. Let $\varphi_R\colon X\to W$ 
be the contraction morphism over $S$ associated to $R$. We put
\[
d = \min_E \dim E, 
\]
where $E$ runs over positive-dimensional irreducible components of $\varphi^{-1}_R(P)$ 
for all $P\in W$. Then $R$ is spanned by a {\em{(}}possibly singular{\em{)}} 
rational curve $\ell$ with
\[
0 < -(K_X+\Delta)\cdot \ell \leq  d+1.
\]
Furthermore, if $\varphi_R$ is birational and $(X, \Delta)$ is kawamata 
log terminal,
then $R$ is spanned by a {\em{(}}possibly singular{\em{)}} rational 
curve $\ell$ with 
\[
0 < -(K_X+\Delta)\cdot \ell <d+1. 
\]

Lastly, let $V$ be an 
irreducible component of the degenerate 
locus 
\[
\bigl\{x \in X \mid \varphi_R  \ \text{is not an isomorphism at}\ x\bigr\}
\] of 
$\varphi_R$. Then $V$ is uniruled.
\end{cor}

\begin{proof}
All we need to do is to replace  
\cite[Theorem~1.12]{fujino6} and \cite[Theorem~1.13]{fujino6}  
with Theorem~\ref{a-thm1.2} and Theorem~\ref{a-thm1.3}, respectively,  
in the proof of \cite[Corollary~12.3]{fujino6} (see \cite[p.~677]{fujino6} for details).  
The desired statement then follows. 
\end{proof}

The following corollary is an immediate consequence of the cone and contraction theorem 
for quasi-log schemes (see \cite[Theorems 6.7.3 and 6.7.4]{fujino3}) 
together with Theorem~\ref{a-thm1.3}.

\begin{cor}[{\cite[Corollary 1.8]{fujino4}}]\label{a-cor1.5}
Let $[X, \omega]$ be a quasi-log canonical pair and let 
$\pi\colon X\to S$ be a projective morphism onto a scheme $S$. 
Let $\mathcal L$ be a $\pi$-ample 
line bundle on $X$. 
We put 
\[
d:=\max_{s\in S} \dim \pi^{-1} (s). 
\]
Then $\omega+(d+1)\mathcal L$ is $\pi$-nef.  
In particular, $\omega+(\dim X+1)\mathcal L$ is always 
$\pi$-nef. 
\end{cor}

Theorem~\ref{a-thm1.2} 
and \cite[Theorem~1.8]{fujino-hashizume} 
(see Conjectures 1.15 and 1.21 in \cite{fujino5}) 
represent the optimal results on this topic in full generality.
For the sake of completeness, we also record here some remarks on 
other statements in \cite{fujino6}.

\begin{rem}\label{a-rem2.1}
In the proof of \cite[Proposition~9.1]{fujino6} (see \cite[p.~665]{fujino6} for details), 
the bound 
\[
0<-(K_X+\Delta)\cdot C \leq 2\dim X
\]
was obtained by using 
\cite[Theorem~1.1]{fujino2}, 
\cite[Theorem~1.12]{fujino6}, or 
\cite[Corollary~12.3]{fujino6}. 
By replacing these results with Theorem~\ref{a-thm1.2} or 
Corollary~\ref{a-cor1.4}, we can improve the bound to 
\[
0<-(K_X+\Delta)\cdot C \leq \dim X+1
\]
in the statement of \cite[Proposition~9.1]{fujino6}.

Since \cite[Theorem~1.8]{fujino6} is an application of 
\cite[Proposition~9.1]{fujino6}, the same improvement applies; 
namely, we may replace
\[
0<-(K_X+\Delta)\cdot C \leq 2\dim X
\]
with
\[
0<-(K_X+\Delta)\cdot C \leq \dim X+1
\]
in \cite[Theorem~1.8]{fujino6} (see \cite[p.~667--668]{fujino6} for details).  

Consequently, because \cite[Theorem~9.2]{fujino6} is derived from 
\cite[Theorem~1.8]{fujino6}, we can similarly replace
\[
0<-\mathcal P\cdot C \leq 2\dim X
\]
with
\[
0<-\mathcal P\cdot C \leq \dim X+1
\] 
(see \cite[pp.~668--669]{fujino4} for details). 

Applying this once more to 
\cite[Theorem~1.6\,(iii)]{fujino6}, which depends on 
\cite[Theorem~9.2]{fujino6}, we obtain the improved inequality
\[
0<-\omega\cdot C_j \leq \dim U_j +1
\]
in place of
\[
0<-\omega\cdot C_j \leq 2\dim U_j 
\] (see \cite[p.~670]{fujino6} for details). 

Finally, since \cite[Theorem~1.5]{fujino6} is a special case of 
\cite[Theorem~1.6]{fujino6}, we may also replace
\[
0<-(K_X+\Delta)\cdot C_j \leq 2\dim U_j
\]
with the sharper bound
\[
0<-(K_X+\Delta)\cdot C_j \leq \dim U_j+1.
\]

Moreover, Theorem~\ref{a-thm1.2} also implies 
\cite[Conjecture~1.21]{fujino6} (see also \cite[Remark~1.22]{fujino6}). 
\end{rem}

\section{A Characterization of Projective Space}\label{b-sec3}

In this section we first prove a weaker analogue of Theorem \ref{thm:maintheorem} for terminal varieties.  We then reduce the general case of Theorem \ref{thm:maintheorem} to the terminal case using the theory of quasi-log schemes.

\subsection{Terminal varieties}

This section focuses on the special case of Theorem \ref{thm:maintheorem} when $X$ has terminal singularities.  We will need the following definition which extends the notion of ``very free'' from rational curves to curves of arbitrary genus.  The key property of very free curves is that they deform with the expected dimension.

\begin{defn}
Let $X$ be a projective variety.  A morphism $s: C \to X$ from a (smooth projective integral) curve $C$ is said to be very free if $s(C)$ lies in the smooth locus of $X$ and every non-zero quotient of $s^{*}T_{X}$ has slope at least $2g(C)+1$.   
\end{defn}

\cite[Theorem 1.3]{JLR25b} guarantees that every terminal Fano variety admits a very free curve in its smooth locus; furthermore, we can ensure that the degree is arbitrarily large compared to the genus.  We prove Theorem \ref{thm:maintheorem} for terminal varieties by studying how very free curves deform and break into rational curves.

The following lemma shows that in our situation a general very free curve of high degree passes through many general points of $X$.

\begin{lem} \label{lem-manyGeneralPoints}
    Let $X$ be a terminal Fano variety of dimension $n$ and suppose $c>0$ is chosen so that $cK_{X}$ is Cartier.  Suppose every rational curve on $X$ has anticanonical degree $>n$. Let $g,d > 0$ be positive integers and define
    \begin{equation*}
        t = \max \left\{ \left\lceil \frac{d}{n} \right\rceil - (cn+1)g - cn, 0 \right\}
    \end{equation*}
    Then every genus $g$ very free curve $f: C \to X$ of anticanonical degree $d$ has a deformation that maps a general set of $t$ points in $C$ to a general set of $t$ points in $X$.
\end{lem}

\begin{proof}
    Since $cK_{X}$ is a Cartier divisor, our assumption implies that every rational curve on $X$ has anticanonical degree $\geq n + \frac{1}{c}$.
    
    Let $M_0$ denote the irreducible component of the moduli space $\mathrm{Mor}(C,X)$ containing our original very free curve $f: C \to X$. 
    We inductively construct a family $M_k$ of curves in the following way. Suppose we have constructed a family $M_i$ sending general points $q_1, \dots, q_i \in C$ to fixed points $p_1, \dots, p_i \in X$.  Select a general point $q_{i+1} \in C$ and let $Y_{i+1} \subset X$ be the subvariety swept out by $f(q_{i+1})$ as $f$ varies over $M_i$. Fix some general $p_{i+1} \in Y_{i+1}$ and let $M_{i+1} \subset M_i$ be the family of curves that send $q_{i+1}$ to $p_{i+1}$. 
    By construction, $\dim M_{i+1} = \dim M_i - \dim Y_{i+1}$.  
    
    Define $\theta(i)$ to be the number of indices $j \leq i$ for which $\dim Y_j < n$.  Because $\dim M_i$ is eventually 0 and $\theta(i)$ is eventually positive, there is a largest integer $k$ such that $\dim M_{k} \geq \theta(k)$.  The family $M_k$ sends $q_1, \ldots, q_k$ to $p_1,\ldots, p_k$ and has dimension between $\theta(k)$ and $\theta(k) + n$.  For ease of notation, set $k_1 = \theta(k)$ and $k_0 = k - \theta(k)$.  

    Because we chose each $q_i$ generally, $\dim Y_i \geq \dim Y_{i+1}$ for all $i \geq 1$.  Thus the set $\{p_1, \ldots , p_{k_0}\}$ is a general collection of $k_0$ points in $X$.  Our goal is to show that $k_0$ is at least $t$.  To do so, consider a resolution of singularities $\pi: X' \to X$ that is an isomorphism over the smooth locus.  Each point $p_i$ lies in the smooth locus of $X$ because it is a general point on a curve that meets the smooth locus.

    We apply \cite[Lemma 2.1]{JLR25a} to the strict transform on $X'$ of the curves parametrized by $M_k$.  This produces a stable map $f': C' \to X'$ where the stabilization of $C'$ is $C$ and the preimages under $C' \to C$ of $q_{k-k_1+1}, \dots, q_k \in C$ are trees of rational curves.  For the rest of the proof, we consider $C$ as the unique irreducible component of $C'$ with positive genus and $q_i \in C \subset C'$ as points of $C'$.

To fix notation, let $p_i' \in X'$ be the unique point that maps to $p_i \in X$ under $\pi$.  Let $I \subset \{1, \ldots , k\}$ be the set of indices $i$ such that $C'$ contains a rational tree attached to $C$ at $q_i$.  Set $h = |I|$.  For each $i \in I$, there is a distinct rational component $R_i \subset C'$ that is not contracted by $f'$ and whose $f'$-image contains $p_i'$.  These rational curves cannot be contracted by $\pi$, so our hypothesis on $X$ implies $$-K_X \cdot \big(\sum_{i \in I} \pi_*f'_*R_i\big) \geq h\left(n+\frac{1}{c}\right).$$
    
We claim that $-K_X \cdot \pi_{*}f'_*C' \geq (k-1)n + \frac{1}{c} k_1$.  If $h = k$, this follows from the previous paragraph because $X$ is Fano.  Otherwise, if $h < k$, observe that any index $i \notin I$ satisfies $i \leq k_0$ by construction.  In particular, $\{p_i'\}_{i \notin I}$ is a general set of $k - h$ points in $X'$.  As $f'(q_i) = p_i'$ for each $i \notin I$, the twist down of $f'^*(T_{X'})|_C$ by all but one of these $k - h$ points is generically globally generated.   Since a generically globally generated vector bundle on $C$ has degree $\geq 0$, we conclude that $-K_{X'} \cdot f'_*C \geq (k -h - 1)n$.  Terminality of $X$ also ensures $-K_{X} \cdot \pi_{*}f'_{*}C \geq -K_{X'} \cdot f'_*C$.  Using the fact that $h \geq k_1$, we obtain
    \begin{align*}
        -K_X \cdot \pi_{*}f'_*C' & \geq -K_X \cdot \left( \pi_*f'_*C + \sum_{i \in I} \pi_*f'_*R_i\right)\\
        & \geq  (k - h - 1) n + h\left(n+\frac{1}{c}\right) \\
        &\geq (k-1)n + \frac{1}{c} k_1.
    \end{align*}
    This proves our claim.
    
    Since the dimension of $M_0$ is $d-(g-1)n$, the dimension of $M_k$ is at most $k_1 + n$, and the codimension of $M_k$ in $M_0$ is at most $n(k-k_1) + (n-1)k_1$, it follows that 
    $$d \leq (g-1)n + k_1 + n + nk - k_1 = (g+k)n.$$
    Putting this together with $(k - 1)n + \frac{1}{c} k_1 \leq d$, we see that
    $$ (k-1)n+\frac{1}{c} k_1 \leq gn + kn $$
    which shows that $k_1 \leq (g+1)nc$.  Using the bound $d \leq (g+k)n$ from above, we find that $k_{0} = k-k_{1} \geq \lceil \frac{d}{n} \rceil - g - (g+1)cn$ as desired.
\end{proof}

    Using Lemma \ref{lem-manyGeneralPoints}, we carefully select a degeneration of a very free curve of large degree.  We show this degeneration breaks our curve into two pieces: a very free rational curve of degree $n+1$ and a higher-genus very free curve. The very free rational curve allows us to apply \cite{CMSB} and conclude that $X$ is $\mathbb{P}^n$.

\begin{thm} \label{thm:vfcurveexists}
    Let $X$ be a terminal Fano variety of dimension $n$ such that every rational curve in $X$ has anticanonical degree  $>n$.  Then $X \simeq \mathbb{P}^{n}$.
\end{thm}

\begin{proof}
    We first prove that $X$ contains a very free rational curve of degree $n+1$ in its smooth locus.  Choose a positive integer $c$ so that $cK_{X}$ is Cartier.  Let $\pi : X' \to X$ be a resolution of singularities that is an isomorphism over the smooth locus.  By \cite[Theorem 1.3]{JLR25b} there is a genus $g \geq 2$ such that $X$ admits a very free curve $C$ of genus $g$ and arbitrarily large anticanonical degree.  In particular, setting $\gamma = (1 + nc)(g + 1)$ we may suppose that the anticanonical degree $d$ of $C$ satisfies
    \begin{equation} \label{eq:degbound}
        d > n(1 + nc)(3g + \gamma). 
    \end{equation}
    Define $s = \lceil \frac{d}{n} \rceil - \gamma$ and $\delta = d-(g-1)n-(s+1)n$.  
    By construction $s \geq 1$ and $\delta \geq 1$.   
    
    Let $q_{0},q_1, \dots, q_{s}, r_1, \dots, r_{\delta-1}$ be general points of $C$. Let $p_0, p_{1}, \dots, p_{s}$ be general points of $X$ and let $D_1, \dots, D_{\delta-1}$ be general divisors in a very ample linear series on $X$. By Lemma \ref{lem-manyGeneralPoints} there is a $\delta$-parameter family of maps $f: C \to X$ with $f(q_i) = p_i$. Thus, we can find a 1-parameter subfamily $M$ of maps $f: C \to X$ with $f(q_i) = p_i$ and $f(r_i) \in D_i$.  We take the strict transforms of these maps and apply Bend-and-Break on $X'$ to obtain a limiting stable map $f': C' \to X'$ where $C'$ has a rational tail attached to $C$ at $q_0$.  Since the image of $q_0$ does not lie in the $\pi$-exceptional locus, we see that the rational tail is not contracted by $\pi \circ f'$.  For the rest of the proof, we consider $C$ as the unique irreducible component of $C'$ with positive genus and consider $q_i, r_j \in C \subset C'$ as points of $C'$.

    We study the map $f': C' \to X'$. Let $R_i$ be the fiber of the stabilization map $C' \to C$ over $q_i$ and let $T_j$ be the fiber of the stabilization map $C' \to C$ over $r_j$. Let $I \subset \{ 1, \dots,  s \}$  be the set of integers $i$ for which $R_i$ is a curve and let $J \subset \{ 1, \dots, \delta-1 \}$ be the set of integers for which $T_j$ is a curve. Set $h = |I|$.
    
    We claim that $s - h \geq 3g$. If not, then $h > s-3g$ and the total $-K_X$-degree of $\pi_*f'_* C'$ must be at least
    \begin{align*}
        (h+1)\left(n+\frac{1}{c}\right) & > (s-3g)\left(n+\frac{1}{c} \right) = sn + \frac{s}{c} - 3g \left(n+\frac{1}{c} \right) \\ & \geq d - \gamma n + \frac{s}{c} -3g\left(n+\frac{1}{c} \right) \geq d +\frac{d}{nc} - (3g + \gamma)\left(n + \frac{1}{c}\right).
    \end{align*}
    By the inequality \eqref{eq:degbound} the right hand side is larger than $d$ and we obtain a contradiction. Thus, $s - h \geq 3g$ and we see that the spine curve  $C \subset C'$ is not contracted by $\pi \circ f'$.  
    
    Note that the deformations of $f'|_C: C \to X'$ which satisfy $f(q_{i}) = p_{i}$ for all but one value of $i \in \{1,\ldots,s\} \backslash I$ define a dominant family of curves.  Thus the twist down of $f'^{*}T_{X'}$ by the sum of the corresponding points $\{ q_{i} \}$ is generically globally generated. Since $s-h-1 \geq 3g-1 \geq 2g+1$ we conclude that every positive rank quotient of $f'^{*}T_{X'}$ has slope at least $2g+1$.  Thus $f'|_C$ is very free and deforms with the expected dimension.

    We now carefully calculate $-K_{X'}$ degrees. Because $f'|_C$ sends $s-h$ chosen points of $C$ to $s-h$ general points in $X'$ and sends $\delta-1 - |J|$ chosen points of $C$ to general divisors in $X$, we see that $h^0(f'^*T_{X'}|_C)$ is at least $(s-h)n+\delta-1-|J|$. Since $f'|_C$ is very free, we know that $h^1(f'^*T_{X'}|_C) = 0$. Hence, by Riemann-Roch the $(-K_{X'})$-degree of $f'_* C$ must be at least $(s-h)n+\delta-1 - |J| +(g-1)n$. 
    
    We next apply \cite[Lemma 4.3]{JLR25b} with $S = R_i$. Since $S$ only meets the rest of $C'$ along the spine curve, which is not contained in the $\pi$-exceptional locus, \cite[Lemma 4.3]{JLR25b} shows that $-K_{X'} \cdot f'_{*}R_i \geq -\pi^*K_X \cdot f'_{*}R_i  > n$. Since it must be an integer, $-K_{X'} \cdot f'_{*}R_i \geq n+1$.   Similarly, Lemma \cite[Lemma 4.3]{JLR25b} shows $-K_{X'} \cdot f'_{*} T_j \geq 1$.  Thus, by \cite[Lemma 4.3]{JLR25b} again the total $-K_{X'}$ degree of $f'_* C'$ is at least 
    \begin{align*}
        -K_{X'} \cdot f'_{*} C + \sum_{i \in I \cup \{ 0 \}} -K_{X'} \cdot f'_{*} R_i + \sum_{j \in J} -K_{X'} \cdot f'_{*} T_j \geq -K_{X'} \cdot f'_{*} C + (h+1)(n+1) + |J| &\\
         \geq (s-h)n+\delta-  1 - |J| + (g-1)n + (h+1)(n+1) +|J| & \\
         = (s+1)n + (g-1)n + \delta +h &\\
         = d + h. &
    \end{align*}
    It follows that $h =0$ and every inequality above is an equality.  Thus, $-K_{X'} \cdot f'_{*}R_0 = n+1$ and $-K_{X'} \cdot f'_{*}C = sn + (g-1)n +\delta-1 - |J|$.  
    
    The map $f' : C' \to X'$ we obtain depends on the choice of each $q_i, p_i, r_j$, and $D_j$, but for general choices the stable map $f'$ will lie in a fixed stratum of the Kontsevich space.  Since $f'|_C$ is very free, the space of deformations of $f'|_C$ has dimension $sn+\delta-1 -|J|$.  Thus there are only finitely many possibilities for the map $f'|_C$ which satisfy the incidence conditions $\pi \circ f'(q_i) = p_i$ for $1 \leq i \leq s$ and $\pi \circ f'(r_j) \in D_j$ for $j \in \{1, \ldots \delta - 1\}\setminus J$.  This means that for general choices of the points $\{ p_{i} \}$ and divisors $\{D_{j}\}$, the map $f'|_{C}$ will be a general element of the component $M \subset \mathrm{Mor}(C,X')$ it lies in.  Moreover, the aforementioned incidence conditions do not depend on the choice of $q_0 \in C$.  As $M$ parameterizes a dominant family of maps and $f'|_C \in M$ is general, the point $f'|_{C}(q_{0}) \in X'$ will be a general point.  In fact, setting $p_0' = \pi^{-1}(p_0)$, the pair of points $(p_0', f'|_{C}(q_{0})) \in X' \times X'$ is general as well, once more because the incidence conditions above are independent of the choice of $p_0$.

    Since $R_0$ attaches to $C$ at $q_0$, $f'(R_0)$ passes through the general pair of points $p_{0}'$ and $f'|_{C}(q_{0})$ of $X'$.  
    Note that $\pi \circ f'(R_0)$ must be irreducible, because otherwise it would have two components both of $(-K_X)$-degree greater than $n$, but 
    $$n + 1 = -K_{X'} \cdot f'_* R_0 \geq -K_X \cdot \pi_*f'_* R_0$$
    by \cite[Lemma 4.3]{JLR25b}.  It follows that $R_0$ must contain an irreducible component $F_0$ such that $f'|_{F_{0}}$ passes through both $p_{0}'$ and $f'|_{C}(q_{0})$; generality of these points implies that $f'|_{F_{0}}$ is a very free rational curve on $X'$ by \cite[4.10]{Debarre01}. 
In particular $-K_{X'} \cdot f'_{*}F_{0} \geq n+1$.  If $R_0$ had more than one irreducible component, we could apply \cite[Lemma 4.3]{JLR25b} with $S$ equal to a connected component of the complement of $F_0$ in $R_0$ to see that $-K_{X'} \cdot f'_*R_0$ is strictly larger than $n+1$,  a contradiction. 
    We conclude that $R_0$ is irreducible. By \cite[Lemma 4.2]{JLR25b}, it follows that $f'(R_0)$ is disjoint from the $\pi$-exceptional locus.

    Thus we have established the existence of a very free rational curve of degree $n+1$ (which by definition means that the image is contained in the smooth locus of $X$). In particular, the irreducible component of $\mathrm{Mor}(\mathbb{P}^{1},X)$ containing such curves is doubly dominant (in the sense of \cite{CMSB}).  Furthermore, the degree assumption implies that the curves on $X$ cannot break as we deform them.  Thus this irreducible component of $\mathrm{Mor}(\mathbb{P}^{1},X)$ is also closed, irreducible, maximal, and everywhere unsplitting.  By \cite[Theorem 4.2]{CMSB} we conclude that $X$ is isomorphic to $\mathbb{P}^{n}$.
\end{proof}

\subsection{General case}

We are now prepared to prove Theorem \ref{thm:maintheorem}:

\begin{proof}[Proof of Theorem~\ref{thm:maintheorem}:]
Let $\varphi_R \colon X \to W$ be the extremal contraction associated to $R$
(see, e.g., \cite[Theorem~1.1]{fujino2}).  
By Corollary~\ref{a-cor1.4} together with the assumption $l(R)>\dim X$,
the target $W$ is a point.  
Consequently, $\rho(X)=1$, the divisor $-(K_X+\Delta)$ is ample, and 
the non-lc locus $\Nlc(X,\Delta)$ satisfies 
$\dim \Nlc(X,\Delta)\le 0$.

Suppose that $(X,\Delta)$ is not kawamata log terminal.  
Assume that there exists a positive-dimensional log canonical center $V$ of $(X,\Delta)$.  
By adjunction (see, e.g., \cite[Theorem 6.3.5 (i)]{fujino3}), 
the pair $[V,\omega]$, where $\omega := (K_X+\Delta)|_V$, is a quasi-log scheme, and 
$-\omega$ is ample.  
By Theorem~\ref{a-thm1.2}, we obtain a rational curve $\ell \subset V$ such that
\[
0 < -\omega \cdot \ell \leq \dim V + 1 .
\]
This yields a rational curve $\ell \subset X$ satisfying
\[
0 < -(K_X+\Delta)\cdot\ell \leq \dim X ,
\]
which contradicts the assumption $l(R)> \dim X$.  
Therefore $(X,\Delta)$ has no positive-dimensional log canonical centers, and
the non-klt locus $\Nklt(X,\Delta)$ is zero-dimensional since 
$\dim \Nlc(X, \Delta)\leq 0$. By \cite[Theorem~1.17]{fujino6} (see also 
\cite[Theorem~1.7]{fujino-hashizume}), 
there exists a rational curve $\ell \subset X$ such that
\[
0 < -(K_X+\Delta)\cdot \ell \le 1 ,
\]
which again contradicts $l(R)>\dim X$.  
Hence $(X,\Delta)$ is kawamata log terminal.

By \cite[Discussion after Corollary 1.4.3]{BCHM10}, there exists a projective birational morphism
$f \colon \widetilde X \to X$ such that $(\widetilde X, \widetilde\Delta)$ is a 
$\mathbb Q$-factorial terminal pair and
\[
K_{\widetilde X} + \widetilde\Delta = f^*(K_X+\Delta).
\]
Note that $-(K_{\widetilde X}+\widetilde\Delta)$ is nef and big.  
Thus we may take a $(K_{\widetilde X}+\widetilde\Delta)$-negative extremal ray 
$\widetilde R \subset \NE(\widetilde X)$ with $f_*\widetilde R = R$.  
Because $f$ is log crepant $l(\widetilde R)> \dim \widetilde X $. 

Let $\varphi_{\widetilde R} \colon \widetilde X \to \widetilde W$ be the associated contraction.  
By Corollary~\ref{a-cor1.4} again, $\widetilde W$ is a point, and hence $\rho(\widetilde X)=1$.  
Therefore $f\colon \widetilde X \to X$ is an isomorphism and $X$ is a $\mathbb Q$-factorial terminal Fano variety.  Furthermore $\Delta$ is nef so that $-K_{X} \cdot C > n$ for every rational curve $C$ on $X$.  By Theorem \ref{thm:vfcurveexists} $X \simeq \mathbb{P}^{n}$. All the remaining desired properties can be verified easily.
\end{proof}

To deduce Corollary \ref{cor:maincor1} and Corollary \ref{coro:maincor2} from Theorem \ref{thm:maintheorem}, we will require the following important remark concerning the value $l(R)$ 
in Theorem~\ref{thm:maintheorem}. 

\begin{rem}\label{b-rem3.3}
Let $(X, \Delta)$ be a projective log canonical pair and let 
$R$ be a $(K_X+\Delta)$-negative extremal ray of 
$\NE(X)$. 
It is well known that 
\begin{equation}\label{b-eq3.1}
l(R)=\min\bigl\{-(K_X+\Delta)\cdot C \, | \, C \textrm{ is rational and }[C]\in R\bigr\},
\end{equation}
that is, there exists a rational curve $C$ on $X$ with $[C]\in R$ such that 
\[
l(R)=-(K_X+\Delta)\cdot C.
\]
If $K_X+\Delta$ is $\mathbb Q$-Cartier, then the equality 
\eqref{b-eq3.1} is immediate. 
When $K_X+\Delta$ is only $\mathbb R$-Cartier, the same conclusion 
follows by adapting the argument in the proof of 
\cite[Theorem~4.7.2\,(1)]{fujino3}. 
\end{rem}

\begin{proof}[Proof of Corollary \ref{cor:maincor1}:]
    Since $(X,\Delta)$ is log canonical, $R$ is relatively ample at infinity and \cite[Theorem~1.1]{fujino2} shows that $R$ is rational.  The statement follows from Theorem \ref{thm:maintheorem} and Remark~\ref{b-rem3.3}.
\end{proof}

\begin{proof}[Proof of Corollary~\ref{coro:maincor2}]
This follows from Corollary \ref{cor:maincor1} and the cone theorem for log canonical pairs.
\end{proof}

Finally, we extend Corollary \ref{cor:maincor1} 
to semi-log canonical pairs.

\begin{cor}\label{cor-slc2}
Let $(X,\Delta)$ be a connected $n$-dimensional projective semi-log canonical pair such that $\NE(X)$ admits a $(K_X+\Delta)$-negative extremal ray $R$ of length greater than $n$.
Then $X\simeq \mathbb P^n$, $(X,\Delta)$ is terminal, and $\Delta$ is a divisor of degree less than $1$.
\end{cor}

\begin{proof}
By the cone and contraction theorem for semi-log canonical pairs, there exists a contraction morphism
\[
\varphi_R \colon X \to W
\]
associated with the extremal ray $R$.
Let $\nu \colon Z \to X$ be the normalization and write
\[
K_Z + \Delta_Z := \nu^*(K_X+\Delta).
\]
Let
\[
Z = \bigsqcup_{i\in I} Z_i
\]
be the decomposition into irreducible components and set $\Delta_{Z_i} := \Delta_Z|_{Z_i}$ for each $i\in I$.
By Theorem~\ref{a-thm1.2}, there exists an index $i_0 \in I$ such that
\[
(\varphi_R \circ \nu)(Z_{i_0})
\]
is a point.
Applying Corollary~\ref{coro:maincor2} to the pair $(Z_{i_0}, \Delta_{Z_{i_0}})$, we conclude that
\[
Z_{i_0} \simeq \mathbb P^n
\quad\text{and}\quad
\lfloor \Delta_{Z_{i_0}} \rfloor = 0.
\]
This forces $Z$ to be irreducible, hence $Z = Z_{i_0}$, and the normalization map $\nu \colon Z \to X$ is an isomorphism.
The assertion now follows from Corollary~\ref{cor:maincor1}.
\end{proof}

\bibliographystyle{alpha}
\bibliography{main}

\end{document}